\documentclass[11pt]{article}
\usepackage{amsmath, amssymb, amsfonts, amsthm, graphicx}

\newtheorem{theorem}{Theorem}[section]
\newtheorem{lemma}[theorem]{Lemma}

\newtheorem{cor}[theorem]{Corollary}
\newtheorem{claim}[theorem]{Claim}

\newtheorem{thm}[theorem]{Theorem}
\newtheorem{preexample}[theorem]{{\bf Example}}
\newenvironment{example}{\begin{preexample}\rm{\hspace{-0.5 em}{\bf}}}{\end{preexample}}

\title{Interval minors of complete bipartite graphs}

\author{
  Bojan Mohar\thanks{Supported in part by an NSERC Discovery Grant (Canada),
   by the Canada Research Chair program, and by the
    Research Grant P1--0297 of ARRS (Slovenia).}~\thanks{On leave from:
    IMFM \& FMF, Department of Mathematics, University of Ljubljana, Ljubljana,
    Slovenia.}\\[1mm]
  Department of Mathematics\\
  Simon Fraser University\\
  Burnaby, BC, Canada\\
  {\tt mohar@sfu.ca}
\and
  Arash Rafiey\\[1mm]
  Department of Mathematics\\
  Simon Fraser University\\
  Burnaby, BC, Canada\\
  {\tt arashr@sfu.ca}
\and
  Behruz Tayfeh-Rezaie\\[1mm]
  School of Mathematics\\
  Institute for Research in Fundamental Sciences (IPM)\\
  P.O. Box 19395-5746, Tehran, Iran\\
  {\tt tayfeh-r@ipm.ir}
\and
 Hehui Wu \\[1mm]
 Department of Mathematics\\
  Simon Fraser University\\
  Burnaby, BC, Canada\\
  {\tt noshellwhh@gmail.com}
}

\date{}

\begin{document}

\maketitle

\begin{abstract}
Interval minors of bipartite graphs were recently introduced by Jacob Fox in the study of Stanley-Wilf limits.
We investigate the maximum number of edges in $K_{r,s}$-interval minor free bipartite graphs. We determine exact values when $r=2$ and describe the extremal graphs. For $r=3$, lower and upper bounds are given and the structure of $K_{3,s}$-interval minor free graphs is studied.
\end{abstract}

\vspace{5mm}

\section{Introduction}

All graphs in this paper are simple, i.e.\ multiple edges and loops
are not allowed.
By an \emph{ordered} bipartite graph $(G; A,B)$, we mean a bipartite graph $G$ with independent sets $A$ and $B$ which partition the vertex set of $G$ and each of $A$ and $B$ has a linear ordering on its elements. We call two vertices $u$ and $v$ \emph{consecutive} in the linear order $<$ on $A$ or $B$ if $u<v$ and there is no vertex $w$ such that $u<w<v$.
By \emph{identifying} two consecutive vertices $u$ and $v$ to a single vertex $w$, we obtain a new ordered bipartite graph such that the neighbourhood of $w$ is the union of the neighbourhoods of $u$ and $v$ in $G$. All bipartite graphs in this paper are ordered and so, for simplicity, we usually say bipartite graph $G$ instead of ordered bipartite graph $(G; A,B)$.
Two ordered bipartite graphs $G$ and $G'$ are \emph{isomorphic} if there is a graph isomorphism $G\to G'$ preserving both parts, possibly exchanging them, and preserving both linear orders. They are \emph{equivalent} if $G'$ can be obtained from $G$ by reversing the orders in one or both parts of $G$ and possibly exchange the two parts.

If $G$ and $H$ are ordered bipartite graphs, then $H$ is called an \emph{interval minor} of $G$ if a graph isomorphic to $H$ can be obtained from $G$ by repeatedly applying the following operations:
\begin{enumerate}
 \item[(i)] deleting an edge;
 \item[(ii)] identifying two consecutive vertices.
\end{enumerate}
If $H$ is not an interval minor of $G$, we say that $G$ \emph{avoids} $H$ as an interval minor or that $G$ is \emph{$H$-interval minor free}.
Let $ex(p,q,H)$ denote the maximum number of edges in a bipartite graph with parts of sizes $p$ and $q$ avoiding $H$ as an interval minor.

In classical Tur\'an extremal graph theory, one asks about the
maximum number of edges of a graph of order $n$ which has no
subgraph isomorphic to a given graph. Originated from problems in
computational and combinatorial geometry, the authors in
\cite{BIE,FUR,FURHAJ} considered Tur\'an type problems for
matrices which can be seen as ordered bipartite graphs. In the
ordered version of Tur\'an theory, the question is: what is the
maximum number edges of an ordered bipartite graph with parts of
size $p$ and $q$ with no subgraph isomorphic to a given ordered
bipartite graph? More results on this problem and its variations
are given in \cite{ALBERT,BRASS,CLA,KLA,PACH}. As another
variation, interval minors were recently introduced by Fox in
\cite{FOX} in the study of Stanley-Wilf limits. He gave
exponential upper and lower bounds for $ex(n,n,K_{\ell,\ell})$. In
this paper, we are interested in the case when $H$ is a complete
bipartite graph. We determine the value of $ex(p,q,K_{2,\ell})$
and find bounds on $ex(p,q,K_{3,\ell})$. We note that our
definition of interval minors for ordered bipartite graphs is
slightly different from Fox's definition for matrices, since we
allow exchanging parts of the bipartition, so for us a matrix and
its transpose are the same. Of course, when the matrix of $H$ is
symmetric, the two definitions coincide.

\section{$K_{2,\ell}$ as interval minor}

For simplicity, we denote $ex(p,q,K_{2,\ell})$ by $m(p,q, \ell)$. In this section we find the exact value of this quantity.
Let $(G;A,B)$ be an ordered bipartite graph where $A$ has ordering $a_1 < a_2 < \dots <a_p$ and $B$ has ordering $b_1 < b_2 <\dots <b_q$.
The vertices $a_1$ and $b_1$ are called \emph{bottom} vertices whereas $a_p$ and $b_q$ are said to be \emph{top} vertices. The degree of a vertex $v$ is denoted by $d(v)$.

\begin{lemma}\label{frtbd}
For any positive integers $p$ and $q$, we have
$$m(p,q,\ell)\leqslant (\ell-1)(p-1)+q.$$
\end{lemma}

\begin{proof}
Let $(G;A,B)$ be a bipartite graph. Suppose that $A$ has ordering $a_1 < a_2 < \dots < a_p$ and $B$ has ordering $b_1 < b_2 < \dots < b_q$.
For $1\leqslant i\leqslant p-1$, let
$$A_i=\{b_j \mid \exists \ i_1\leqslant i < i_2 \text{ such that } a_{i_1}b_j,a_{i_2}b_j\in E(G)\}.$$
Since $G$ is $K_{2,\ell}$-interval minor free, $|A_i|\leqslant \ell-1$.
 Each $b_j\in B$ appears in at least
$d(b_j)-1$ of sets $A_i$, $1\leqslant i\leqslant p-1$. It follows that
\begin{equation*}
\sum_{i=1}^q (d(b_j)-1) \leqslant \sum_{i=1}^{p-1} |A_i| \le (\ell-1)(p-1).
\end{equation*}
This proves that $|E(G)|\leqslant(\ell-1)(p-1)+q$.
\end{proof}

If $(G;A,B)$ and $(G'; A',B')$ are disjoint ordered bipartite graphs and the bottom vertices $x,y$ of $G$ are adjacent and the top vertices $x',y'$ of $G'$ are adjacent, then we denote by $G\oplus G'$ the ordered bipartite graph obtained from $(G\cup G'; A\cup A', B\cup B')$ by identifying $x$ with $x'$ and $y$ with $y'$, where the linear orders of $A\cup A'$ and $B\cup B'$ are such that the vertices of $G'$ precede those of $G$. The graph $G\oplus G'$ is called the \emph{concatenation} of $G$ and $G'$.

In the description of $K_{2,\ell}$-interval minor free graphs below, we shall use the following simple observation, whose proof is left to the reader.
Let $G$ and $G'$ be vertex disjoint $K_{r,s}$-interval minor free bipartite graphs with $r\ge2$ and $s\ge2$ such that the bottom vertices in $G$ are adjacent and the top vertices in $G'$ are adjacent. Then $G\oplus G'$ is $K_{r,s}$-interval minor free.

\begin{example}\label{ex1}
We introduce a family of $K_{2,\ell}$-interval minor free bipartite graphs which would turn out to be extremal.
Let $\ell\geqslant 3$ and let $p$ and $q$ be positive integers and let $r = \lfloor (p-1)/(\ell-2)\rfloor$ and $s = \lfloor (q-1)/(\ell-2)\rfloor$. We can write $p = (\ell-2)r+e$ and $q = (\ell-2)s+f$, where $1 \leqslant e \leqslant \ell-2$, $1 \leqslant f \leqslant \ell-2$. Suppose now that $r<s$.
Let $H_0$ be $K_{e,\ell-1}$ and let $H_i$ be a copy of $K_{\ell-1,\ell-1}$ for $1 \leqslant i \leqslant r$.
The concatenation $H=H_0\oplus H_1\oplus \cdots \oplus H_r$ is $K_{2,\ell}$-interval minor free by the above observation.
It has parts of sizes $p$ and $q'=(\ell-2)(r+1)+1$. It also has $r\ell(\ell-2)+e(\ell-1)$ edges.
Finally, let $H^+ = K_{1,q-q'+1}$. The graph $\mathcal{H}_{p,q}(\ell) = H^+\oplus H$ has parts of sizes $p,q$ and has $(\ell-1)(p-1)+q$ edges.
An example is depicted in Figure \ref{ex2f}(b), where the identified top and bottom vertices used in concatenations are shown as square vertices.
\end{example}

\begin{figure}[htb]
\centering
\includegraphics[width=12.8cm]{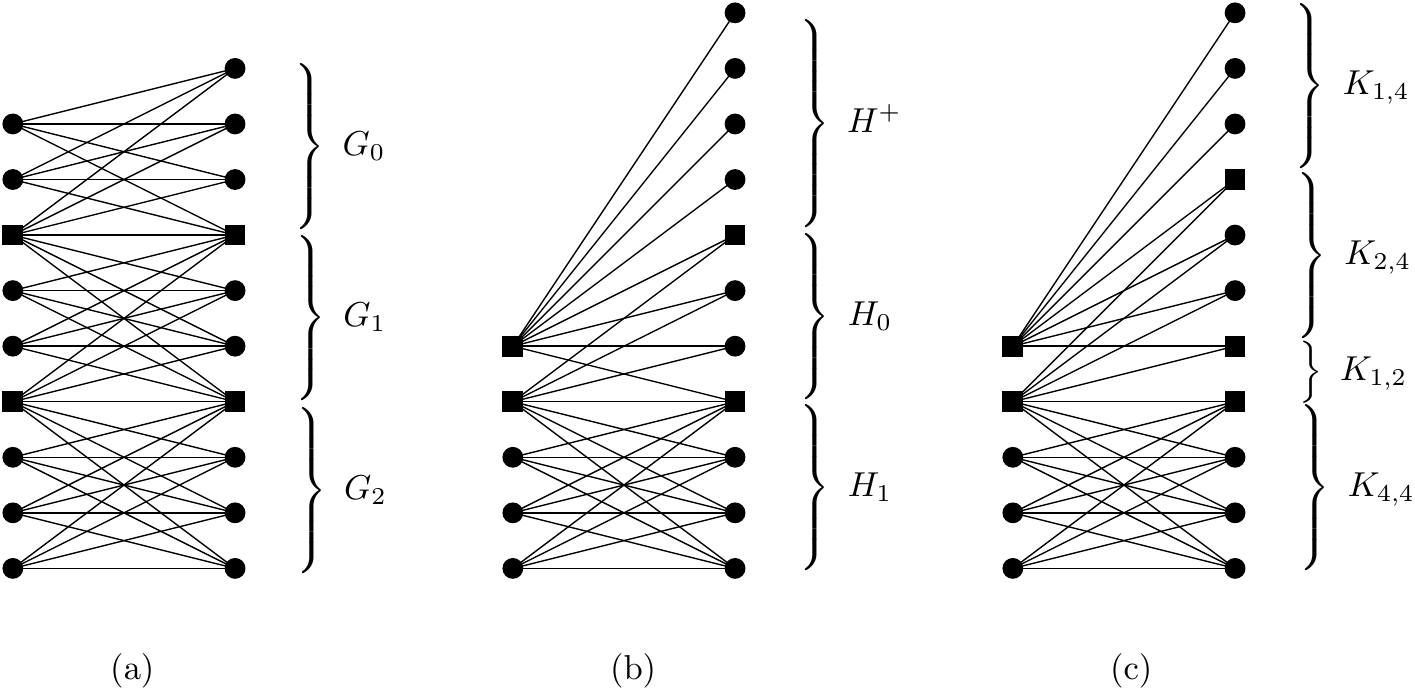}
\caption{(a) $\mathcal{G}_{9,10}(5)$, (b) $\mathcal{H}_{5,11}(5)$,
         (c) $K_{1,4} \oplus K_{2,4} \oplus K_{1,2} \oplus K_{4,4}$}\label{ex2f}
\end{figure}

By Lemma \ref{frtbd} and Example \ref{ex1}, the following is obvious.

\begin{thm}\label{thm:r<s}
Let $\ell\geqslant 3$, $p = (\ell-2)r+e$ and $q = (\ell-2)s+f$, where $1 \leqslant e \leqslant \ell-2$, $1 \leqslant f \leqslant \ell-2$. If\/ $r<s$, then
$$m(p,q,\ell)=(\ell-1)(p-1)+q.$$
\end{thm}

Extremal graphs for excluded $K_{2,\ell}$ given in Example \ref{ex1} are of the form of a concatenation of $r$ copies of $K_{\ell-1,\ell-1}$ together with $K_{e,\ell-1}$ and $K_{1,t}$ where $t=q-(\ell-2)(r+1)$. Note that the latter graph itself is a concatenation of copies of $K_{1,2}$ and that the constituents concatenated in another order than given in the example, are also extremal graphs. For an example, consider the graph in Figure \ref{ex2f}(c), which is also extremal for $(p,q,\ell)=(5,11,5)$.
Rearranging the order of concatenations is not the only way to obtain examples of extremal graphs. What one can do is also using the following operation. Delete a vertex in $B$ of degree 1, replace it by a degree-1 vertex $x$ adjacent to any vertex $a_i\in A$ which is adjacent to two consecutive vertices $b_j$ and $b_{j+1}$, and put $x$ between $b_j$ and $b_{j+1}$ in the linear order of $B$. This gives other extremal examples that cannot always be written as concatenations of complete bipartite graphs.

And there is another operation that gives somewhat different extremal examples. Suppose that $G$ is an extremal graph for $(p,q,\ell)$ with $r<s$ as above. If $A$ contains a vertex $a_i$ of degree $\ell-1$ (by Theorem \ref{thm:r<s}, degree cannot be smaller since the deletion of that vertex would contradict the theorem), then we can delete $a_i$ and obtain an extremal graph for $(p-1,q,\ell)$. The deletion of vertices of degrees $\ell-1$ can be repeated. Or we can delete any set of $k$ vertices from $A$ if they are incident to precisely $k(\ell-1)$ edges.

We now proceed with the much more difficult case, in which we have $\lfloor (p-1)/(\ell-2)\rfloor = \lfloor (q-1)/(\ell-2)\rfloor$, i.e.\ $r=s$.

\begin{example} \label{ex2}
Let $\ell\geqslant 3$, $p = (\ell-2)r+e$ and $q = (\ell-2)r+f$, where $1 \leqslant e \leqslant \ell-2$ and $1 \leqslant f \leqslant \ell-2$. Similarly as in Example \ref{ex1}, let $G_0$ be $K_{e,f}$ and let $G_i$ be a copy of $K_{\ell-1,\ell-1}$ for $1 \leqslant i \leqslant r$. Let $\mathcal{G}_{p,q}(\ell)$ be the concatenation $G_0\oplus G_1\oplus \cdots \oplus G_r$.
This graph is $K_{2,\ell}$-interval minor free.
It has parts of sizes $p,q$ and has $r\ell(\ell-2)+ef$ edges. An example is illustrated in Figure \ref{ex2f}(a).
\end{example}

\begin{thm}\label{rs2}
Let $\ell\geqslant 3$,
$p= (\ell-2)r+e$ and $q= (\ell-2)r+f$, where $1 \leqslant e \leqslant \ell-2$ and $1 \leqslant f \leqslant \ell-2$. Then
$$m(p,q,\ell)=r\ell(\ell-2)+ef.$$
\end{thm}

\begin{proof}
Since the graphs in Example \ref{ex2} attain the stated bound, it suffices to establish the upper bound, $m(p,q,\ell)\leqslant r\ell(\ell-2)+ef$.
Let $(G;A,B)$ be a bipartite graph with parts of sizes $p,q$ and with $m(p,q,\ell)$ edges. Let $A$ have ordering $a_1 < a_2 < \dots <a_p$ and $B$ have ordering $b_1 < b_2 <\dots <b_q$.
Note that any two consecutive vertices of $G$ have at least one
common neighbour. Otherwise, by identifying two consecutive vertices with no common neighbour lying say in $A$, we obtain a graph with parts of sizes $p-1,q$ and with $m(p,q,\ell)$ edges. This is a contradiction since clearly $m(p,q,\ell)>m(p-1,q,\ell)$.

For $1\leqslant i\leqslant p-1$, let
$$A_i=\{b_j \mid \exists \ i_1\leqslant i < i_2 \text{ such that } a_{i_1}b_j,a_{i_2}b_j\in E(G)\}.$$
Also let $A'_i=A_i\setminus \{b_h\}$, where $h$ is the smallest index for which $b_h\in A_i$.
Since $G$ is $K_{2,\ell}$-interval minor free, $|A'_i|\leqslant \ell-2$.
For each vertex $b_j\in B$, define
$$D(b_j)=\{a_i \mid j \text{ is the smallest index such that } a_i \text{ is adjacent to } b_j\},$$ and
let $d'(b_j)=|D(b_j)|$. Every vertex in $N(b_j)\setminus D(b_j)$ is adjacent to $b_j$ and also to some vertex $b_h\in B$ with $h<j$ and hence
\begin{equation} \label{dajdj}
 d(b_j)-d'(b_j)\leqslant \ell-1
\end{equation}
since $G$ is $K_{2,\ell}$-interval minor free.
Let $h$ and $h'$ be the smallest and largest indices such that $a_h,a_{h'} \in N(b_j)\setminus D(b_j)$, respectively. Observe that $h'-h \geqslant d(b_j)-d'(b_j)-1$.
We claim that $b_j$ appears in sets $A'_{h},A'_{h+1},\dots,A'_{h'-1}$. Let $h \leqslant i < h'$.
Since $b_j$ is adjacent to $a_h$ and to $a_{h'}$, we have $b_j\in A_{i}$.
We know that $a_h$ is adjacent to some vertex $b_{j_1}$ with $j_1<j$. Also
$a_{h'}$ is adjacent to some vertex $b_{j_2}$ with $j_2<j$. Suppose that $j_1 \leqslant j_2$. Now we use the property that every two consecutive vertices of $G$ have at least one
common neighbour for consecutive pairs of vertices $b_t,b_{t+1}$ ($t=j_1,\dots,j_2-1$). It follows that there is $j_1\leqslant j_0\leqslant j_2$ such
that $b_{j_0}$ is in $A_{i}$. If $j_2<j_1$, the same property used for $t=j_2,\dots,j_1-1$ shows that there exists $j_0$, $j_2\leqslant j_0\leqslant j_1$, such
that $b_{j_0} \in A_i$.
Since $j_0<j$, from the definition of $A'_{i}$, we conclude that $b_j\in A'_{i}$. So we have proved the claim. We conclude that
$b_j$ appears in sets $A'_h,A'_{h+1},\dots,A'_{h+t-1}$ for some $1\leqslant h\leqslant p-1$ and $t=d(b_j)-d'(b_j)-1$.

Let $S=\{i \mid 1\leqslant i\leqslant p-1, i\not\equiv 1, \ldots,e-1 \pmod{\ell-2}\}$. We have $|S|= r(\ell-1-e)$. By the conclusion in the last paragraph, each $b_j\in B$ appears in at least
$d(b_j)-d'(b_j)-1$ consecutive sets $A'_i$. Combined with (\ref{dajdj}), we conclude that $b_j$ appears in at least
$d(b_j)-d'(b_j)-1-(e-1)$ of sets $A'_i$, where $i\in S$. Note that this number is negative for
$j=1$ since $d(b_1)=d'(b_1)$. Now it follows that
\begin{equation}\label{eq:2}
\sum_{i=2}^q (d(b_j)-d'(b_j)-e) \leqslant \sum_{i\in S} |A'_i|.
\end{equation}
By adding $d(b_1)-d'(b_1)$ to the left side of (\ref{eq:2}) and noting that $\sum_j d(b_j) = |E(G)|$ and $\sum_j d'(b_j) = p$, we obtain therefrom that
$$|E(G)|-p-eq+e\leqslant r(\ell-1-e)(\ell-2).$$
This in turn yields that
$|E(G)|\leqslant r\ell(\ell-2)+ef$, which we were to prove.
\end{proof}

Example \ref{ex2} describes extremal graphs for Theorem \ref{rs2}. They are concatenations of complete bipartite graphs, all of which but at most one are copies of $K_{\ell-1,\ell-1}$. If $e=1$ and $f>1$, vertices of degree 1 can be inserted anywhere between two consecutive neighbors of their neighbor in $A$. But in all other cases, we believe that all extremal graphs are as in Example \ref{ex2}, except that the order of concatenations can be different.

\section{$K_{2,2}$ as interval minor}

In this section we determine the structure of $K_{2,2}$-interval minor free bipartite graphs.
We first define two families of $K_{2,2}$-interval minor free graphs.
For every positive integer $n\geqslant 3$, let $A=\{x,a_1,\ldots,\allowbreak a_{n-1},z\}$ and $B=\{b_1,y,b_2',b_2,\ldots,b_{n-1},b_{n-1}',t,b_{n}\}$ with ordering $x<a_1<\cdots<a_{n-1}<z$ and $b_1<y<b_2'<b_2<\cdots<b_{n-1}'<t<b_{n}$, respectively. Let $R_n$ be the bipartite graph with parts $A,B$ and edge set
$$
   E(G)=\{a_ib_{i},a_ib_{i+1} \mid 1\leqslant i\leqslant n-1\}\cup\{xy,a_1b_2',a_{n-1}b_{n-1}',zt\}.
$$
Similarly we define a graph $S_n$ for every integer $n\geqslant 2$. Let
$A=\{x,a_1,\ldots,\allowbreak a_{n-1},a_{n-1}',z,a_n\}$ and $B=\{b_1,y,b_2',b_2,\ldots,b_{n},t\}$ with ordering $x<a_1<\cdots<a_{n-1}'<z<a_n$ and $b_1<y<b_2'<b_2<\cdots<b_{n}<t$, respectively. Let $S_n$ be the bipartite graph with parts $A,B$ and edge set
$$E(G)=\{a_ib_{i},a_ib_{i+1} \mid 1\leqslant i\leqslant n-1\}\cup\{xy,a_1b_2',a_{n-1}'b_n,zt,a_nb_n\}.$$
For instance, $R_5$ and $S_4$ are shown in Figure \ref{figsn}.

\begin{figure}[htb]
\centering
\includegraphics[width=8.5cm]{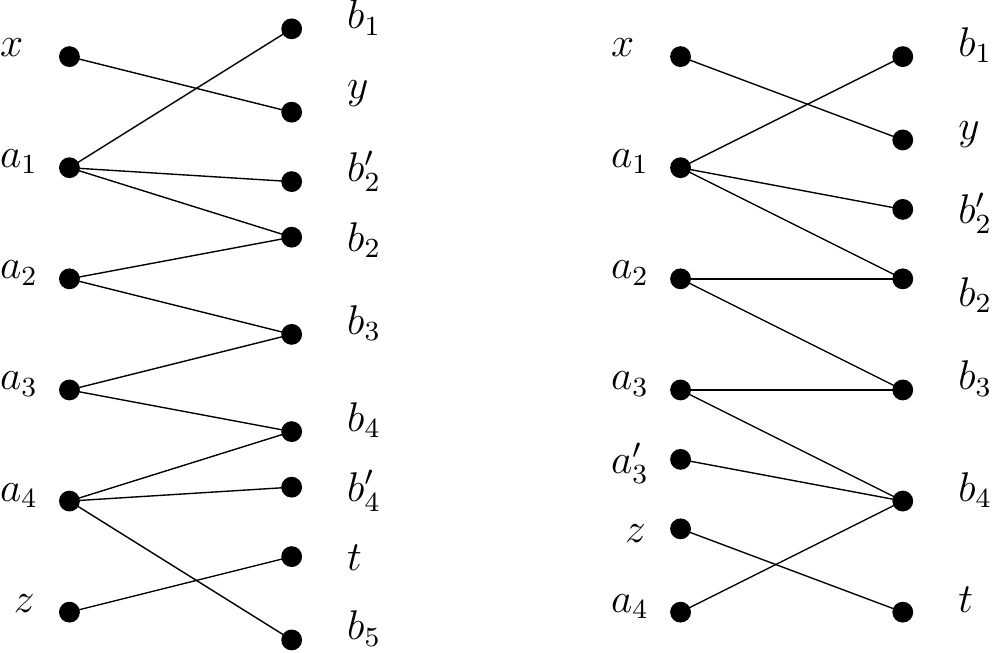}
\caption{The graphs $R_5$ and $S_4$.}\label{figsn}
\end{figure}

\begin{lemma}
For every positive integers $p$ and $q$, we have
$m(p,q,2)=p+q-1$.
\end{lemma}

\begin{proof}
By Lemma \ref{frtbd}, $m(p,q,2)\leqslant p+q-1$. We construct $K_{2,2}$-interval minor free bipartite graphs with parts of sizes $p,q$ and with $p+q-1$ edges. This is easy if
$p\leqslant 4$. So let $5\leqslant p\leqslant q$. Consider $S_{p-3}$ and add edges $a_1y, zb_{p-3}$. Also add $q-p$
vertices into the set $B$, all of them ordered between $y$ and $b_2'$, and join each of them to $a_1$. The resulting graph has parts of size $p,q$ and has $p+q-1$ edges.
\end{proof}

In what follows we assume that $(G;A,B)$ is a bipartite graph without $K_{2,2}$ as an interval minor. Let $A$ and $B$ have the ordering $a_1 < a_2 < \dots <a_p$ and $b_1 < b_2 <\dots <b_q$, respectively. A vertex in $G$ of degree 0 is said to be \emph{reducible}. If $d(a_i)=1$ and the neighbor $b_j$ of $a_i$ is adjacent to $a_{i-1}$ if $i>1$ and is adjacent to $a_{i+1}$ if $i<p$, then $a_i$ is also said to be \emph{reducible}. Similarly we define when a vertex $b_j\in B$ is reducible. Clearly, if $a_i$ (or $b_j$) is reducible, then $G$ has a $K_{2,2}$-interval minor if and only if $G-a_i$ ($G-b_j$) has one. Therefore, we may assume that we remove all reducible vertices from $G$. When $G$ has no reducible vertices, we say that $G$ is \emph{reduced}, which we assume henceforth.

Let $X=\{a_1,a_2\}$ if $d(a_1)=1$ and $X=\{a_1\}$, otherwise. Similarly, let
$Y=\{a_{p-1},a_p\}$ if $d(a_p)=1$ and $Y=\{a_p\}$, otherwise; $Z=\{b_1,b_2\}$ if $d(b_1)=1$ and $Z=\{b_1\}$, otherwise;
$T=\{b_{q-1},b_q\}$ if $d(b_q)=1$ and $T=\{b_q\}$, otherwise.
We may assume that all these sets are mutually disjoint. Otherwise $G$ has a simple
structure -- it is equivalent to a subgraph of a graph shown in Figure \ref{fig:XY} and any such graph has no $K_{2,2}$ as interval minor. Note that each such subgraph becomes equivalent to a subgraph of $R_2$ after removing reducible vertices.

\begin{figure}[htb]
\centering
\includegraphics[width=2cm]{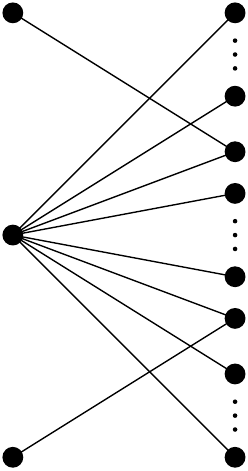}
\caption{$X$ and $Y$ intersect only in special situations.}\label{fig:XY}
\end{figure}

\begin{claim}\label{Xadj}
There is an edge from $X$ to $b_1$ or $b_q$.
\end{claim}

\begin{proof}
Suppose that there is no edge from $X$ to $\{b_1,b_q\}$.
Since $G$ is reduced, there are two distinct vertices $b_i$ and $b_j$ $(1<i<j<q)$
connected to $X$. Assume that $b_1$ and $b_q$ are adjacent to $a_k$ and $a_l$, respectively.
Note that $a_k,a_l\not\in X$. Consider the sets $X$, $A\setminus X, \{b_1,\ldots,b_i\}$ and $\{b_{i+1},\ldots,b_q\}$ and identify them to single vertices to get $K_{2,2}$ as an interval minor, a contradiction.
\end{proof}

Note that Claim \ref{Xadj} also applies to $Y, Z$ and $T$.
Hence, considering an equivalent graph of $G$ instead of $G$ if necessary, we may assume that there is an edge from $X$
to $Z$. If there is no edge from $Y$ to $T$, then there are edges from $Y$ to $Z$ and from $T$ to $X$. By reversing the order of $B$, we obtain an equivalent graph that has edges from $X$ to $Z$ and from $Y$ to $T$. Thus we may assume henceforth that the following claim holds:

\begin{claim}\label{cl:XZYT}
The graph $G$ has edges from $X$ to $Z$ and from $Y$ to $T$.
\end{claim}

\begin{claim}\label{eachdeg}
Every vertex of $G$ has degree at most $2$, except possibly one of $a_2,b_2$ and/or one of $a_{p-1},b_{q-1}$, which may be of degree $3$. If $d(a_2)=3$, then it has neighbors $b_1,b_3,b_4$, we have $d(a_1)=d(b_1)=d(b_2)=1$ and $a_1b_2\in E(G)$. Similar situations occur when $b_2$, $a_{p-1}$, or $b_{q-1}$ are of degree $3$.
\end{claim}

\begin{proof}
Suppose that $d(a_i)\ge3$. We claim that $a_i$ has at most one neighbour in $Z$.
Otherwise, $|Z|\ge2$ and hence $d(b_1)=1$ and $a_ib_1,a_ib_2\in E(G)$. This is a contradiction since $G$ is reduced. Similarly we see that $a_i$ has at most one neighbour in $T$.

Suppose now that a middle neighbor $b_j$ of $a_i$ is in $B\setminus (Z\cup T)$. Let $b_{j_1}$ and $b_{j_2}$ be neighbors of $a_i$ with $j_1<j<j_2$. If $d(b_j)>1$, then an edge $a_kb_j$ ($k\ne i$), the edges joining $X$ and $Z$ and joining $Y$ and $T$, and the edge $a_ib_{j_1}$ (if $k<i$) or $a_ib_{j_2}$ (if $k>i$) can be used to obtain a $K_{2,2}$-interval minor. Thus, $d(b_j)=1$.

Let us now consider $b_{j-1}$.
Suppose that $b_{j-1}$ is not adjacent to $a_i$. Then $j_1<j-1$. If $b_{j-1}$ is adjacent to a vertex $a_k$, where $k<i$, then the edges $a_ib_{j_1}, a_kb_{j-1}$ and the edges joining $X$ with $Z$ and $Y$ with $T$ give rise to a $K_{2,2}$-interval minor in $G$ (which is excluded), unless the following situation occurs: the edge $a_kb_{j-1}$ is equal to the edge joining $X$ and $Z$. This is only possible if $j_1=1$, $j=3$ and $|Z|=2$, i.e., $d(b_1)=1$. If $a_1$ is adjacent to $b_1$ or to some other $b_t$ with $t>2$, we obtain a $K_{2,2}$-interval minor again. So, it turns out that $k=1$ and $d(a_1)=1$. If $i>2$, then we consider a neighbor of $a_2$. It cannot be $b_2$ since then $a_1$ would be reducible. It can neither be $b_1$ or $b_t$ with $t>2$ since this would yield a $K_{2,2}$-interval minor. Thus $i=2$.

Similarly, a contradiction is obtained when $k>i$. (Here we do not have the possibility of an exception as in the case when $k=1$.) Thus, we conclude that $b_{j-1}$ is adjacent to $a_i$ or we have the situation that $i=2$, $j=3$, etc.\ as described above. Similarly we conclude that $b_{j+1}$ is adjacent to $a_i$ unless we have $i=p-1$, $j=q-2$, etc. Note that we cannot have the exceptional situations in both cases at the same time since then we would have $i=2=p-1$ and $X\cap Y$ would be nonempty. If $a_ib_{j-1}$ and $a_ib_{j+1}$ are both edges, then $b_j$ would be reducible, a contradiction. Thus, the only possibility for a vertex of degree more than 2 is the one described in the claim.
\end{proof}

\begin{claim}\label{neiba}
We have $a_1$ adjacent to $b_1$ or we have $a_1$ adjacent only to $b_2$ and $b_1$ adjacent only to $a_2$.
\end{claim}

\begin{proof}
Suppose that $a_1b_1\notin E(G)$.
By Claim \ref{cl:XZYT}, $X$ is adjacent to $Z$ and $Y$ to $T$. If $X$ is adjacent to a vertex $b_j\notin Z$ and $Z$ is adjacent to a vertex $a_i\notin X$, then we have a $K_{2,2}$-interval minor in $G$. Thus, we may assume that $X$ has no neighbors outside $Z$. Since $a_1b_1\notin E(G)$, we have that $a_1b_2\in E(G)$. In particular, $d(a_1)=1$ and $d(b_1)=1$. Then $a_2\in X$ and $b_2\in Z$. Since $G$ is reduced, $a_2b_2\notin E(G)$. Since all neighbors of $X$ are in $Z$, we conclude that $a_2$ is adjacent to $b_1$. This yields the claim.
\end{proof}

The same argument applies to the bottom vertices.

We can now describe the structure of $K_{2,2}$-interval minor free graphs. In fact, we have proved the following theorem.

\begin{thm}\label{maink22}
Every reduced bipartite graph with no $K_{2,2}$ as an interval minor is equivalent to a subgraph of $R_n$ or $S_n$ for some positive integer $n$.
\end{thm}

A matching of size $n$ is a 1-regular bipartite graph on $2n$ vertices. The following should be clear from Theorem \ref{maink22}.

\begin{cor}
For every integer $n \ge 4$, there are exactly eight $K_{2,2}$-interval minor free matchings of size $n$. They form three different equivalence classes.
\end{cor}

\section{$K_{3,\ell}$ as interval minor}

For $K_{3,\ell}$-interval minors in bipartite graphs, we start in a similar manner as when excluding $K_{2,\ell}$. We first establish a simple upper bound, which will later turn out to be optimal in the case when the sizes of the two parts are not very balanced.

\begin{lemma}\label{frtbd3}
For any integers $\ell\ge1$ and $p,q\ge2$, we have
$$ex(p,q,K_{3,\ell})\leqslant (\ell-1)(p-2)+2q.$$
\end{lemma}

\begin{proof}
Let $(G;A,B)$ be a bipartite graph with parts of sizes $p$ and $q$. Suppose that $A$ has ordering $a_1 < a_2 < \dots <a_p$ and $B$ has ordering $b_1 < b_2 <\dots <b_q$.
For $2\leqslant i\leqslant p-1$, let
$$A_i=\{b_j \mid a_ib_j\in E(G), \exists \ i_1 < i < i_2 \text{ such that } a_{i_1}b_j,a_{i_2}b_j\in E(G)\}.$$
If $G$ is $K_{3,\ell}$-interval minor free, we have $|A_i|\leqslant \ell-1$.
Each $b_j\in B$ of degree at least 2 appears in precisely
$d(b_j)-2$ of the sets $A_i$, $2\leqslant i\leqslant p-1$. It follows that
\begin{equation*}
\sum_{j=1}^q (d(b_j)-2) \leqslant \sum_{i=2}^{p-1} |A_i|.
\end{equation*}
This gives
$|E(G)|\leqslant(\ell-1)(p-2)+2q$, as desired.
\end{proof}

Let $(G;A,B)$ and $(G'; A',B')$ be disjoint ordered bipartite graphs. Let $a_{p-1},a_p$ be the last two vertices in the linear order in $A$ and let $b_{q-1},b_q$ be the last two vertices in $B$. Denote by $a_1',a_2'$ and $b_1',b_2'$ the first two vertices in $A'$ and $B'$, respectively. Let us denote by $G\oplus_2 G'$ the ordered bipartite graph obtained from $G$ and $G'$ by identifying $a_{p-1}$ with $a_1'$, $a_p$ with $a_2'$, $b_{q-1}$ with $b_1'$, and $b_q$ with $b_2'$. The resulting ordered bipartite graph $G\oplus_2 G'$ is called the \emph{$2$-concatenation} of $G$ and $G'$.
We have a similar observation as used earlier for $K_{2,\ell}$-free graphs.
If $a_{p-1},a_p$ and $b_{q-1},b_q$ form $K_{2,2}$ in $G$ and $a_1',a_2'$ and $b_1',b_2'$ form $K_{2,2}$ in $G'$, and $r\ge3$ and $s\ge3$, then $G\oplus_2 G'$ is $K_{r,s}$-interval minor free if and only if $G$ and $G'$ are both $K_{r,s}$-interval minor free.

\begin{example} \label{ex3}
Let $\ell\geqslant 4$,
$p = (\ell-3)r+e$ and $q = (\ell-3)s+f$ where $2 \leqslant e \leqslant \ell-2$, $2 \leqslant f \leqslant \ell-2$ and $r<s$.
Let $\mathcal{K}_{p,q}(\ell)$ be the 2-concatenation of $K_{e,\ell-1}$, $r$ copies of $K_{\ell-1,\ell-1}$ and $K_{2,q-(\ell-3)(r+1)}$. This graph has parts of sizes $p$ and $q$ and has $(\ell-1)(p-2)+2q$ edges.
\end{example}

By Lemma \ref{frtbd3} and Example \ref{ex3}, the following is clear.

\begin{thm}
Let $\ell\geqslant 4$,
$p = (\ell-3)r+e$ and $q = (\ell-3)s+f$ where $2 \leqslant e \leqslant \ell-2$, $2 \leqslant f \leqslant \ell-2$ and $r<s$. Then
$$ex(p,q,K_{3,\ell})=(\ell-1)(p-2)+2q.$$
\end{thm}

We now consider the remaining cases, where both parts are ``almost balanced'', i.e., $\lfloor (p-2)/(\ell-3) \rfloor = \lfloor (q-2)/(\ell-3) \rfloor$.

\begin{example} \label{ex4}
Let $\ell\geqslant 4$,
$p = (\ell-3)r+e$ and $q = (\ell-3)r+f$ where $2 \leqslant e \leqslant \ell-2$ and $2 \leqslant f \leqslant \ell-2$.
Let $\mathcal{K}_{p,q}(\ell)$ be the 2-concatenation of $K_{e,f}$ and $r$ copies of $K_{\ell-1,\ell-1}$. This graph is $K_{3,\ell}$-interval minor free, has parts of sizes $p$ and $q$, and has $r(\ell-3)(\ell+1)+ef$ edges.
It follows that
$$ex(p,q,K_{3,\ell}) \geqslant r(\ell-3)(\ell+1)+ef.$$
We conjecture that this is in fact the exact value for $ex(p,q,K_{3,\ell})$.
Unfortunately, we have not been able to adopt the proof of Theorem \ref{rs2} for this case.
\end{example}

\end{document}